\documentclass[a4paper, 12pt]{amsart}

\usepackage{amssymb}

\mathchardef\emptyset="001F

\theoremstyle{plain}

\newtheorem{theorem}{Theorem}[section]
\newtheorem{lemma}[theorem]{Lemma}

\newtheorem{corollary}[theorem]{Corollary}

\newtheorem{definition}[theorem]{Definition}

\theoremstyle{definition}
\theoremstyle{remark}
\numberwithin{equation}{section}

\newcommand{\Div}{\mathrm{Div}}
\newcommand{\e}{\varepsilon}

\newcommand{\weak}{\rightharpoonup}

\newcommand{\R}{{\mathbb R}}

\renewcommand{\S}{{\mathcal S}}

\newcommand{\T}{{\mathbb T}}
\newcommand{\M}{{\mathbb M}}
\newcommand{\Mmn}{\M^{{m\times n}}}
\newcommand{\Mqt}{\M^{4\times 3}}
\renewcommand{\S}{{\mathbb S}}

\newcommand{\no}{\noindent}

\newcommand{\diag}{\hbox{{\rm diag}}}
\renewcommand{\a}{\alpha}

\newcommand{\dsp}{\displaystyle}
\newcommand{\rank}{\mathrm{rank}}

\def\XXint#1#2#3{{\setbox0=\hbox{$#1{#2#3}{\int}$}
\vcenter{\hbox{$#2#3$}}\kern-.5\wd0}}

\title[]{On the relationship between rank-$(n-1)$ convexity and 
${\mathcal S}$-quasiconvexity}
\author[Mariapia Palombaro]{Mariapia Palombaro}

\begin{document}
\baselineskip3.15ex
\vskip .3truecm

\maketitle
\small{
\center{SISSA}
\vspace{-1mm}
\center{Via Beirut 2-4, 34014 Trieste, Italy}
\vspace{-1mm}
\center{Email: palombar@sissa.it}
\center{}
}

\vspace{2mm}

\begin{abstract}
\small{We prove that rank-$(n-1)$ convexity does not imply 
${\mathcal S}$-quasiconvexity (i.e., 
quasiconvexity with respect to divergence free fields) in $\M^{m\times n}$ 
for $m>n$, by adapting the well-known 
\v Sver\'ak's counterexample \cite{sver91} to the solenoidal setting.
On the other hand, we also remark that 
rank-$(n-1)$ convexity and ${\mathcal S}$-quasiconvexity
turn out to be equivalent in the space of $n\times n$ diagonal matrices. 
This follows by a generalization of M\"uller's work \cite {muller}. 

}
\end{abstract}

\section{Introduction}
\no
The purpose of this note is to generalize some known results about the 
relationship  
between rank-one convexity and quasiconvexity to the context of divergence free 
fields. 
This is motivated by the lower semicontinuity results provided by 
Fonseca and M\"uller (\cite{fonmu}, Theorems 3.6-3.7).
Let us recall the relevant definitions.
A function $f:\M^{m\times n}\to\R$ on the $m\times n$ matrices is called 
rank-one convex if it is convex on each rank-one line, i.e., for every $A,Y\in\M^{m\times n}$ 
with rank$(Y)=1$, the function $\dsp t\to f(A+tY)$ is convex. 
It is quasiconvex if 
\begin{equation*}
\int_{\T^{n}} f(A+\nabla\varphi)\,dx\geq f(A)\,,
\end{equation*}
\no
for all $A\in\M^{m\times n}$ and for all $\T^{n}$-periodic functions $\varphi\in C^{\infty}(\R^{n},\R^{m})$, where $\T^{n}:=(0,1)^{n}$.
Quasiconvexity implies rank-one convexity. 
Whether the converse is true for $m=2$ and $n\geq 2$ is an outstanding open 
problem. In the higher dimensional case $m\geq 3$, 
\v Sver\'ak's counterexample \cite{sver91} shows that rank-one convexity is not 
the same as quasiconvexity. On the other hand, M\"uller \cite {muller} proved 
that the two notions are equivalent for $2\times 2$ diagonal matrices. 
See also \cite{chaudmull}, \cite{lmm08} for further generalizations.

In the spirit of ${\mathcal A}$-quasiconvexity (see, e.g., \cite{fonmu}), 
we provide in this note the counterpart of these results in the context of 
divergence free fields. 
The corresponding notion of quasiconvexity for solenoidal 
fields, that we call ${\mathcal S}$-quasiconvexity, is defined as follows. 

 
\begin{definition}\label{defi} 
A function $f:\M^{m\times n}\to \R$ is said to be 
${\mathcal S}$-quasiconvex 
if for each smooth $\T^{n}$-periodic matrix field 
$B:\R^{n}\to\M^{m\times n}$ such that $\Div B=0$,  
the following inequality holds: 
\begin{equation}\label{disug}
\int_{\T^{n}} f(B)\,dx\geq f\left(\int_{\T^{n}} B\,dx\right)\,.
\end{equation}
\end{definition}

\no
The symbol $\Div$ in the Definition \ref{defi} denotes the operator which acts as 
the divergence on each row of the matrix field $B$. 
\no
While quasiconvexity implies convexity along rank-one lines, 
it is easily checked that 
${\mathcal S}$-quasiconvexity implies convexity along $\rank$-$(n-1)$ lines. Indeed if 
a function $f$ is ${\mathcal S}$-quasiconvex, then $t\to f(A+tV)$ is convex 
for every $A,V\in\M^{m\times n}$ with $\rank(V)\leq n-1$. 
Our aim in this note is to show that 
$\rank$-$(n-1)$ convexity does not imply 
${\mathcal S}$-quasiconvexity in $\M^{m\times n}$, for $m\geq n+1\geq 4$. 
More precisely we prove the following result.

\begin{theorem}\label{dadimo}
For all $n\geq 3$ and $m\geq n+1$, there exists 
$F:\Mmn\to \R$ such that $F$ is rank-$(n-1)$ convex 
but not ${\mathcal S}$-quasiconvex.
\end{theorem}

\no
The proof of Theorem \ref{dadimo} is essentially based on the 
\v Sver\'ak's counterexample adapted to the solenoidal setting and is 
reminiscent of an example given by Tartar in connection with a theorem in 
compensated compactness (see \cite{tartar}, pp. 185--6). 

We do not know whether ${\mathcal S}$-quasiconvexity and 
$\rank$-$(n-1)$ convexity are equivalent in the case when $m=n$. 
Nevertheless M\"uller's result on quasiconvexity on diagonal matrices 
extends as well to the divergence free fields.   
If we identify the space $D(n)$ of diagonal $n\times n$ matrices with $\R^{n}$ 
via $\dsp{y}\to \diag(y_{1},\dots,y_{n})$, then 
a rank-$(n-1)$ convex function on $D(n)$ may be regarded as a 
function on $\R^{n}$ which is convex on each hyperplane 
$\{y_{i}=\rm{const}\}$, $i=1,\dots,n$. 
Then, a straightforward generalization of Theorem 1 in \cite{muller} 
(see also Theorem 1.5 in \cite{lmm08})
asserting that 
rank-one convexity implies quasiconvexity on diagonal matrices, leads to the 
following statement assuring that 
rank-$(n-1)$ convexity implies ${\mathcal S}$-quasiconvexity on diagonal matrices. 

\begin{theorem}\label{teomuller}
Let $1<p<\infty$. Assume that $f:\R^{n}\to\R$ is convex  on each hyperplane 
$\{y_{i}=\rm{const}\}$, $i=1,\dots,n$, and satisfy  
$0\leq f(y)\leq C(1+|y|^{p})$. Suppose that 
\begin{align*}
& u_{h}^{i}\weak u_{\infty}^{i}\,,\quad\:
  \text{ in } L^{p}_{\rm loc}(\R^{n}) \: \text{ as } h\to\infty\,,
  \quad i=1,\dots,n\,, \\
& \partial_{i}u_{h}^{i}\to \partial_{i}u_{\infty}^{i}\,,\quad\:
  \text{ in } W^{-1,p}_{\rm loc}(\R^{n})\: \text{ as } h\to\infty\,,
  \quad i=1,\dots,n. 
\end{align*}
Then for every open set $V\subset\R^{n}$
\begin{equation*}
\int_{V}f(u_{\infty}^{1},\dots,u^{n}_{\infty})\,dx \leq
\liminf_{h\to\infty}\int_{V}f(u_{h}^{1},\dots,u_{h}^{n})\,dx \,.
\end{equation*}
\end{theorem}

\no
We remark that, for $n=2$, Theorem \ref{teomuller} reduces itself to  
Theorem 1 in \cite{muller} (in \cite{muller} the case $n=p=2$ is considered; 
see Theorem 1.5 in \cite{lmm08} for generalization to any $n\geq 2$ and 
$1<p<\infty$). 
Indeed, in dimension two, the notion of 
${\mathcal S}$-quasiconvexity coincides with that of quasiconvexity since any 
divergence free field defines a gradient field upon left multiplication by the 
rotation $\textstyle{\left(\begin{array}{lr}0 & -1\\ 1 & 0\end{array}\right)}$. 
Therefore \v Sver\'ak's example \cite{sver91} shows that the conclusion of 
Theorem \ref{dadimo} already holds for $n=2$.

Finally we recall that another result in the direction of Theorem \ref{teomuller} is 
that if $f$ is a quadratic form and is $\rank$-$(n-1)$ convex, then $f$ is also 
${\mathcal S}$-quasiconvex (see \cite{tartar2}).

\section{Proof of Theorem \ref{dadimo}}
\no
Theorem \ref{dadimo} is a consequence of Lemma \ref{mainth} and Corollary 
\ref{coroll} below.
We will basically follow \v Sver\'ak's strategy. 
The key idea is to find three rank-$(n-1)$ directions such that 
these directions are the only rank-$(n-1)$ directions in the vector space spanned by them, which we call $L$. 
Then one defines a rank-$(n-1)$ convex function on $L$ and 
seeks a divergence free field that takes values only in $L$ and for which the 
inequality \eqref{disug} is violated. The desired function $F$ is then obtained 
by suitably extending the rank-$(n-1)$ convex function defined on $L$ to 
the whole space. 
We first construct an example in $\Mqt$ and then we will extend it to $\Mmn$.
Let  $V_{1},V_{2},V_{3}\in\Mqt$ be given by 

\begin{equation}\label{rank2dir}
V_{1}=\left(
\begin{array}{lll}
1 & 0 & 0 \\
0 & 1 & 0 \\
0 & 0 & 0 \\
0 & 0 & 0 
\end{array}
\right) \,,\quad 
V_{2}=\left(
\begin{array}{lll}
0 & 1 & 0 \\
0 & 0 & 0 \\
0 & 0 & 1 \\
0 & 0 & 0 
\end{array}
\right) \,,\quad 
V_{3}=\left(
\begin{array}{lll}
0 & 0 & 0 \\
0 & 0 & 0 \\
0 & 1 & 0 \\
1 & 1 & 1 
\end{array}
\right) \,.
\end{equation}
\no
We consider the three-dimensional subspace of $\Mqt$ generated by $V_{1},V_{2},V_{3}$: 

\begin{equation}\label{sub}
L:={\rm span}\{V_{1},V_{2},V_{3}\}=\{ \eta_{1} V_{1}+\eta_{2}V_{2}+\eta_{3}V_{3} \,,  \eta_{1},\eta_{2},\eta_{3} \in\R\}\,,
\end{equation}
\no
and we define the function $f: L\to \R$ in the following way
\begin{equation}\label{func}
\forall\, \eta_{1},\eta_{2},\eta_{3}\in\R \quad \quad
f(\eta_{1}V_{1}+\eta_{2}V_{2}+\eta_{3}V_{3})=-\eta_{1}\eta_{2}\eta_{3}\,.
\end{equation}
\no
It can be checked that the only rank-two directions in $L$ are given by $V_{1},V_{2},V_{3}$ 
and therefore the function $f$ is convex (in fact linear) on each rank-two line contained in $L$.

\begin{lemma}\label{lemmasv}
Let $L$ and $f$ be defined by \eqref{sub} and \eqref{func} respectively and let 
${\displaystyle P:\Mqt\to L}$ be the orthogonal projection onto $L$.  
Then for each $\e>0$ there exists $k=k(\e)>0$ such that the function 
$F:\Mqt\to\R$ given by 
\begin{equation}\label{ext}
F(X)=f(PX) + \e |X|^{2} + \e |X|^{4} + k|X-PX|^{2}
\end{equation} 
is rank-two convex on $\Mqt$.
\end{lemma}

\no
Lemma \ref{lemmasv} is an obvious extension of Lemma 2 in \cite{sver91} and 
therefore we refer the reader to \cite{sver91} for its proof.  
We remark that an extension of the form \eqref{ext} is always possible if  $V_{1},V_{2},V_{3}$ are any three rank-$(n-1)$ directions in $\Mmn$ such that 
they are the only rank-$(n-1)$ directions in the subspace spanned by them and 
$f$ is defined as in \eqref{func}.
 
\begin{lemma}\label{mainth}
There exist $\e>0$ and $k>0$ such that the function $F$ given by \eqref{ext} is rank-two 
convex but not ${\mathcal S}$-quasiconvex.
\end{lemma}

\begin{proof}
Let $B:\T^{3}\to \Mqt$ be defined by 
\begin{equation*}
B(x)=
\left(
\begin{array}{ccc}
\cos 2\pi x_{3}           & \cos 2\pi x_{1}            & 0 \\
   0                            & \cos 2\pi x_{3}            & 0 \\
   0                            & \cos 2\pi (x_{1}-x_{3}) & \cos 2\pi x_{1} \\
\cos 2\pi (x_{1}-x_{3}) & \cos 2\pi (x_{1}-x_{3}) &  \cos 2\pi (x_{1}-x_{3})
\end{array}
\right)\,.
\end{equation*}
\no
It is readily seen that the matrix field $B$ defined above is divergence-free and it satisfies  

\begin{equation*}
\left\{
\begin{array}{l}
\displaystyle 
B(x)\in L\quad \forall\, x\in\R^{3}  \,,\\
[3mm]
\displaystyle
\int_{\T^{3}}B\, dx = 0 \,, \\   
[3mm]
\displaystyle
\int_{\T^{3}}f(B)\, dx = 
- \int_{\T^{3}}(\cos 2\pi x_{1})^{2}(\cos 2\pi x_{3})^{2}\,dx <0 \,.  
\end{array}
\right.
\end{equation*}
\no
Since $B$ is bounded, we can choose $\e>0$ such that 
\begin{equation}\label{ineq} 
\int_{\T^{3}}\Big(f(B) + \e |B|^{2} + \e |B|^{4}\Big)  \, dx <0 \,.
\end{equation}
By Lemma \ref{lemmasv} there exists $k=k(\e)$ such that the function 
\begin{equation*}
F(X)=f(PX) + \e |X|^{2} + \e |X|^{4} + k|X-PX|^{2}
\end{equation*} 
is rank-two convex. Since $|B(x)-PB(x)|=0$ for all $x$ in $\R^{3}$, 
we have from \eqref{ineq} 
\begin{equation*} 
\int_{\T^{3}} F(B(x))  \, dx <0 \,,
\end{equation*}
which concludes the proof.
\end{proof}

\begin{corollary}\label{coroll}
For all $n> 3$ and $m\geq n+1$, there exists 
$F^{(n)}:\Mmn\to \R$ such that $F$ is rank-$(n-1)$ convex 
but not ${\mathcal S}$-quasiconvex  in  $L^{p}(\R^{n},\M^{m\times n})$, 
for every $p\geq 1$.
\end{corollary}

\begin{proof}
We show how to adapt the counterexample constructed in Lemma \ref{mainth} 
to an arbitrary dimension $n$. 
Since one can always increase the number of rows by 
adding some zeros while preserving the rank of the matrices, 
it is enough to consider 
the case when $m=n+1$. 
In this situation we will exhibit three matrices 
$V_{1}^{(n)},V_{2}^{(n)},V_{3}^{(n)}$ which satisfy the following properties 
\begin{align}
\label{prop1} 
& \rank (V_{i}^{(n)})\leq n-1 \quad \forall \, i=1,2,3\,, \\
\label{prop2} &\rank(\a_{1}V_{1}^{(n)}+\a_{2}V_{2}^{(n)}+\a_{3}V_{3}^{(n)})=n 
\quad \forall \, \a\in \S^{2}\setminus\{(\pm 1,0,0),(0,\pm 1,0),(0,0,\pm 1)\}\,.  
\end{align} 
For each $\a$ in $\S^{2}$ we set 
$$
M^{(n)}(\a):=\a_{1}V_{1}^{(n)}+\a_{2}V_{2}^{(n)}+\a_{3}V_{3}^{(n)}\,.
$$
We first consider the case when $n=4$. 
We define $V_{1}^{(4)},V_{2}^{(4)},V_{3}^{(4)}$ as follows
\begin{equation}\label{dim4}
V_{1}^{(4)}=\left(
\begin{array}{llll}
1 & 0 & 0 & 0\\
0 & 1 & 0 & 0\\
0 & 0 & 0 & 0\\
0 & 0 & 0 & 1\\
0 & 0 & 0 & 0
\end{array}
\right) \,,\quad 
V_{2}^{(4)}=\left(
\begin{array}{llll}
0 & 1 & 0 & 0\\
0 & 0 & 0 & 0\\
0 & 0 & 1 & 0\\
0 & 0 & 0 & 0\\
0 & 0 & 0 & 0
\end{array}
\right) \,,\quad 
V_{3}^{(4)}=\left(
\begin{array}{llll}
0 & 0 & 0 & 0\\
0 & 0 & 0 & 0\\
0 & 1 & 0 & 0\\
1 & 1 & 1 & 0\\
0 & 0 & 0 & 1
\end{array}
\right) \,.
\end{equation}
We have that $\rank(V_{1}^{(4)})=\rank (V_{3}^{(4)})=3$ and 
$\rank(V_{2}^{(4)})=2$.
In order to see that the condition \eqref{prop2} is satisfied 
it is convenient to write the explicit formula for  
$M^{(4)}(\a)$:
$$
M^{(4)}(\a)=
\left(
\begin{array}{cccc}
\a_{1} & \a_{2} & 0        & 0\\
0        & \a_{1} & 0         & 0\\
0        & \a_{3} & \a_{2} & 0\\
\a_{3} & \a_{3} & \a_{3} & \a_{1}\\
0        & 0         & 0         & \a_{3}
\end{array}
\right) \,.
$$ 
Observe that the $4\times 3$ minor of $M^{(4)}(\a)$ which is obtained 
eliminating the fifth row and the fourth column is a linear combination 
of the matrices  $V_{1},V_{2},V_{3}$ defined by \eqref{rank2dir}.
Then using the fact that $V_{1},V_{2},V_{3}$ satisfy \eqref{prop2} 
for $n=3$, one easily checks that $\rank(M^{4}(\a))=4$. 
Remark that replacing the entry 
$M^{(4)}_{4,4}(\a)=\a_{1}$ by $M^{(4)}_{4,4}(\a)=\a_{2}$  
would give another possible 
choice of $V_{1}^{(4)},V_{2}^{(4)},V_{3}^{(4)}$.

For $n=5$ we choose $V_{1}^{(5)},V_{2}^{(5)},V_{3}^{(5)}$ 
such that $M^{(5)}(\a)$ is given by
$$
M^{(5)}(\a)=
\left(
\begin{array}{ccccc}
\a_{1} & \a_{2} & 0        & 0        & 0     \\
0        & \a_{1} & 0         & 0        & 0     \\
0        & \a_{3} & \a_{2} & 0         & 0     \\
\a_{3} & \a_{3} & \a_{3} & a_{44}       & 0     \\
0        & 0         & 0         & \a_{3} & a_{55}    \\
0        & 0         & 0         & 0         &\a_{3} 
\end{array}
\right) 
$$ 
where $a_{ii}$ can be chosen in the set $\{\a_{1},\a_{2}\}$. 
Proceeding in a similar way, for every $n\geq 4$ we define the matrix $M^{(n)}(\a)$ such that 
\begin{align*}
& M^{(n)}(\a)\in \M^{(n+1)\times n}\,, \\ 
& M^{(n)}_{i,j}(\a)=M^{(n-1)}_{i,j}(\a) \quad 
   \text{ for } i\leq n, j\leq n-1\,, \\
& M^{(n)}_{n,n}(\a)=a_{nn} 
   \quad\text{ where } a_{nn}\in \{\a_{1},\a_{2}\}\,,\\
& M^{(n)}_{n+1,n}(\a))= \a_{3} \,, \\
& M^{(n)}_{i,j}(\a)=0 \quad \text{ otherwise }.
\end{align*}

\no
By construction we have that $\rank(M^{(n)}(\a))=n$  
for all $\a\in \S^{2}\setminus\{(\pm 1,0,0),(0,\pm 1,0),(0,0,\pm 1)\}$.  
For each $n$ we set $L^{(n)}:=
{\rm span}\{V_{1}^{(n)},V_{2}^{(n)},V_{3}^{(n)}\}$ 
and we define the function $f^{(n)}:L^{(n)}\to\R$ as in \eqref{func}, i.e., 
$f^{(n)}(\eta_{1}V_{1}^{(n)}+\eta_{2}V_{2}^{(n)}+\eta_{3}V_{3}^{(n)}):=
-\eta_{1}\eta_{2}\eta_{3}$. 
The sought function $F^{(n)}$ is then defined as in \eqref{ext} with 
$f^{(n)}$ in the 
place of $f$ and with $P$ the orthogonal projection onto $L^{(n)}$.
Considering the divergence free field 
$B^{(n)}(x):=\cos(2\pi x_{3})V_{1}^{(n)}+\cos(2\pi x_{1})V_{2}^{(n)}+
\cos2\pi (x_{1}-x_{3})V_{3}^{(n)}$, we see that $F^{(n)}$ is 
not ${\mathcal S}$-quasiconvex.
\end{proof}

\medskip
\centerline{\sc Acknowledgements}
\vspace{2mm}

\no
I am grateful to Enzo Nesi for drawing my attention to this problem  
and for fruitful discussions. 
This work was done during my post-doc at the Max Planck Institute 
for Mathematics in the Sciences in Leipzig.

\end{document}